\documentclass[11pt]{amsart}
\usepackage{amssymb}
\usepackage{amsmath}
\usepackage[hidelinks]{hyperref}
\usepackage{etoolbox}

\newtheorem{theorem}{Theorem}[section]
\newtheorem{lemma}[theorem]{Lemma}

\newtheorem{proposition}[theorem]{Proposition}
\newtheorem{definition}[theorem]{Definition}
\newtheorem{corollary}[theorem]{Corollary}
\newtheorem{conjecture}[theorem]{Conjecture}

\theoremstyle{remark}
\newtheorem{remark}[theorem]{Remark}

\numberwithin{equation}{section}

\newcommand{\qp}{\vskip .12cm}
\newcommand{\hp}{\vskip .2cm}
\newcommand{\p}{\vskip .4cm}

\newcommand{\Z}{\mathbb{Z}}
\newcommand{\Q}{\mathbb{Q}}
\newcommand{\A}{\mathcal{A}}
\newcommand{\B}{\mathcal{B}}
\newcommand{\X}{\mathcal{X}}
\newcommand{\Y}{\mathcal{Y}}
\newcommand{\CZ}{\mathcal{Z}}
\newcommand{\CO}{\mathcal{O}}
\newcommand{\til}{\tilde}
\newcommand{\bsl}{\backslash}
\newcommand{\ra}{\rightarrow}
\newcommand{\xra}{\xrightarrow}
\newcommand{\depth}{\operatorname{depth}}

\newcommand{\CH}{\mathcal{H}}
\newcommand{\sG}{\mathsf{G}}
\newcommand{\sS}{\mathsf{S}}
\newcommand{\uG}{\underline{G}}
\newcommand{\uS}{\underline{S}}

\newcommand{\Lie}{\operatorname{Lie}}

\newcommand{\ad}{\text{ad}}
\newcommand{\Lg}{\mathfrak{g}}
\newcommand{\Lsl}{\mathfrak{sl}_2}

\begin{document}

\title{Inductive structure of Shalika germs and affine Springer fibers}
\author{Cheng-Chiang Tsai}
\begin{abstract}
This article has two parallel perspectives: to demonstrate an inductive structure of Shalika germs, and to show an analogous inductive structure for affine Springer fibers. More precisely, we give an algorithm to compute arbitrary Shalika germs (resp. affine Springer fibers up to stratification) in terms of three ingredients: Shalika germs (resp. affine Springer fibers) for twisted Levi subgroups, a finite list of combinatorial objects, and the numbers of rational points on varieties over the residue field (resp. varieties themselves) among an explicit finite list of such. We also discuss some formal applications of the algorithm to Shalika germs and orbital integrals.
\end{abstract}
\makeatletter
\patchcmd{\@maketitle}
  {\ifx\@empty\@dedicatory}
  {\ifx\@empty\@date \else {\vskip3ex \centering\footnotesize\@date\par\vskip1ex}\fi
   \ifx\@empty\@dedicatory}
  {}{}
\patchcmd{\@adminfootnotes}
  {\ifx\@empty\@date\else \@footnotetext{\@setdate}\fi}
  {}{}{}
\makeatother
\maketitle 

\tableofcontents

\section{Introduction}
Let $F$ be a non-archimedean local field and $k$ be its residue field. Let $G=\underline{G}(F)$ where $\underline{G}$ is a connected reductive group over $F$. We will always assume $\mathrm{char}(k)$ is large enough. In particular all subtori of $G$ should be tamely ramified. No attempt shall be made to achieve a minimal assumption on $\mathrm{char}(k)$.\p

Denote by $\Lg=\text{Lie }G$ and by $C_c^{\infty}(\Lg)$ the space of compactly supported locally constant complex-valued functions on $\Lg$. Let $\CO^G(0)$ be the set of nilpotent orbits in $\Lg$. With appropriate assumptions on $\text{char}(k)$, and thus $\text{char}(F)$, this is a finite set. DeBacker's enhanced version \cite{De02a} of the theorem of Shalika \cite{Sh72} states that for any element $\gamma\in\Lg$, there exist constants $\Gamma_{\CO}(\gamma)\in\Q$ for all $\CO\in\CO^G(0)$ and a lattice $\Lambda_{\gamma}\subset\Lg$ such that the normalized orbital integral $I^G(\gamma,f)$ satisfies
\begin{equation}\label{Shalika}
I^G(\gamma,f)=\sum_{\CO\in\CO^G(0)}\Gamma_{\CO}(\gamma)I^G(\CO,f),
\end{equation}
for any $f\in C_c^{\infty}(\Lg)$ that are locally constant by $\Lambda_{\gamma}$. The numbers $\Gamma_{\CO}(\gamma)$ are called (normalized) {\bf Shalika germs}. Here the the adjective ``normalized'' is the $|D^G(X)|^{1/2}$ factor as in \cite[Sec. 13.12 and 17.11]{Ko05}. Our first main result is an algorithm to compute any Shalika germ for $\gamma\in\Lg$ in terms of the following three ingredients:\hp

(i)$\;\;$ Shalika germs for twisted Levi subgroups. Recall $G'\subset G$ is called a twisted Levi subgroup if, after base change to a tamely ramified extension, it becomes a Levi subgroup.\qp
(ii)$\;$ A finite list (determined by $G$) of combinatorial numbers that essentially arise from nilpotent orbital integrals and/or Weyl groups. We do not try to understand these numbers.\qp
(iii) Geometric numbers which are the numbers of $k$-points on certain varieties over $k$. These varieties come from another finite list (determined by $G$) of families of varieties, described as certain quasi-finite covers of Hessenberg varieties in Section \ref{secgeom}.\p

Results of DeBacker \cite{De02a} and \cite{De02b} are foundational to our method. After that, an application of an idea of Kim and Murnaghan (Lemma \ref{KM}) is used to reduce the problem for $G$ to its twisted Levi subgroups.\p

On the other side, by ``geometrizing'' the method in \cite[Sec. 2]{De02a} (part of which appeared in work of Waldspurger \cite{Wald95}), we are able to translate our algorithm for Shalika germs to one for affine Springer fibers. Recall that when $F=k((t))$, a regular semisimple affine Springer fibers \cite{KL88} is a reduced, locally of finite type scheme over $k$ such that counting its $k$-points amounts to a coset-counting problem that's equivalent to a certain orbital integral for $G\curvearrowright\Lg$. Our second main result (Theorem \ref{Amherst}) is then a combinatorial algorithm which stratifies an arbitrary affine Springer fiber, so that each stratum is an iterated smooth fibration over some slight generalization of affine Springer fibers for a twisted Levi subgroup. In the fibers of these fibrations we see varieties on which we count $k$-points in the Shalika germs computation.\p

Hales had predicted some inductive properties of Shalika germs (see Corollary \ref{homo}) at least around the time of \cite{Ha94}. We also make the remark that our inductive structure of Shalika germs and affine Springer fibers is close to the inductive construction of tame supercuspidal representations of Yu \cite{Yu01}. It will be nice to have a similar result for characters, extending something along the direction of \cite{KM06} or the direction of \cite{AS09}.\p

This article is structured as follows. In Section \ref{secBT} we introduce notations as well as necessary lemmas for all the rest of the article. The article is then split into two logically independent parts: Section \ref{secASF} and all others. In Section \ref{secASF}, the algorithm to inductively understand affine Springer fibers is presented.\p

For the other part, in Section \ref{secmain} we give the algorithm for computing Shalika germs, leaving it to Section \ref{secgeom} for an explicit description of the varieties on which we have to count $k$-points. After that, we give three formal applications of the algorithm in Section \ref{secapp}: a local constancy result and a homogeneity result near a singular semisimple element for Shalika germs, and a uniform bound for orbital integrals on characteristic functions of parahorics.\p

\section*{Acknowledgment} I thank Julee Kim and Fiona Murnaghan for explaining to me their ideas in \cite{KM03} and \cite{KM06}, which had really inspired this work. I want to express my gratitude to Thomas Hales who shared with me many of his fundamental ideas on Shalika germs. The completion of this paper benefits from discussions with Alexis Bouthier, Jessica Fintzen, Alexei Oblomkov, and Zhiwei Yun.\p

\p
\section{Preparation in Bruhat-Tits theory}\label{secBT}
We continue the notations from the introduction and introduce a bit more. However note that in Section \ref{secmain}, \ref{secgeom} and \ref{secapp}, $F$ is a non-archimedean local field while in Section \ref{secASF} we have $F=k((t))$ for some $k=\bar{k}$. In either case, $F$ (and thus $F$-points of groups over $F$) is given the usual analytic topology defined by its valuation. Our $k$ is always the residue field of $F$, and we assume $\mathrm{char}(k)$ is sufficiently large, in a way determined by the absolute root datum of $\underline{G}$.\p

Let $\B_G$ be the (extended) Bruhat-Tits building of $G$. For any $x\in\B_G$, let $G_x$ be the stabilizer of $x$ in $\B_G$ (which contains the parahoric with finite index). For any $d\in\Q$, denote by $\Lg_{x,d}$ the Moy-Prasad filtration lattice \cite{MP96}, and by $\Lg_{x,d+}:=\Lg_{x,d+\epsilon}$ for very small $\epsilon>0$. The subgroup $G_{x,0+}$ is similarly defined. We'll use the notation $\Lg_{x,d:d+}:=\Lg_{x,d}/\Lg_{x,d+}$, and write $\sG_x$ the reductive quotient at $x$, i.e. the reductive group over $k$ for which $\sG_x(k)\cong G_{x,0:0+}$ ($\sG_x$ unfortunately might be disconnected).\p

Fix $x\in\B_G$ with rational coordinates. Write $m$ the least common multiple of the coordinates, by which we mean $m$ is the smallest integer such that $d\not\in\frac{1}{m}\Z\Rightarrow\Lg_{x,d:d+}=0$. Assume $\text{char}(k)\nmid m$. Then Reeder and Yu gave a description of $\Lg_{x,d:d+}$ as follows\hp

\begin{lemma}\cite[Thm. 4.1]{RY14}\label{RY} Fix a choice of $m$-th root of unity $\zeta_m\in\bar{k}^{\times}$. There exists a connected reductive group $\bar{G}$ over $k$ together with an order $m$ automorphism $\theta$ on $\bar{G}\otimes\bar{k}$ such that we have natural isomorphisms
\[\sG_x\cong\bar{G}^{\theta}.\]
and for every $n\in\Z$,
\[\Lg_{x,\frac{n}{m}:\frac{n}{m}+}\cong\bar{\Lg}^{(n)}(k),\]
where we put $\bar{\Lg}=\text{Lie }\bar{G}$, and $\bar{G}^{\theta}$ as well as $\bar{\Lg}^{(n)}:=\bar{\Lg}^{\theta=\zeta_m^n}$ are defined over $k$.
\end{lemma}\hp

We will talk about regularity, semisimplicity, nilpotency, etc, of elements in $\Lg_{x,\frac{n}{m}:\frac{n}{m}+}$ by identifying them with their images in $\bar{\Lg}$. We nevertheless remark that semisimplicity (resp. nilpotency) of elements in $\bar{\Lg}^{(n)}$ agrees with the usual notion that their $\bar{G}^{\theta}$-orbit is closed (resp. its closure contains $0$), see \cite[Prop. 1,2,3]{Vi76} and their proofs. The group $\bar{G}$ is obtained by taking the reductive quotient at $x$ after base change to a tamely ramified extension of ramification index $m$ so that $x$ becomes hyperspecial. We note that graded Lie algebras were largely studied as initiated by Vinberg \cite{Vi76}.\p

For our convenience we fix a maximal split torus $\uS\subset\uG$ and denotes by $S=\uS(F)$ and $\A\subset\B_G$ the corresponding apartment. Following the spirit of DeBacker \cite{De02b}, we define\p

\begin{definition}\label{nildat} We define a {\bf nilpotent datum} to be a triple $(x,d,e)$, where $x\in\A$, $d\in\Q$ with denominator non-zero in $k$, and $e\in\Lg_{x,d:d+}$, so that $e$ is nilpotent in the sense of Lemma \ref{RY}. We consider $(x,d,e)$ and $(x,d,e')$ to be the same datum if $e$ and $e'$ are in the same $G_{x,0:0+}$-orbit in $\Lg_{x,d:d+}$.
\end{definition}\hp

Under the setup after Lemma \ref{RY}, there are only finitely many nilpotent $\sG_x$-orbits in $\bar{\Lg}^{(n)}$; one shows that a nilpotent orbit in $\bar{\Lg}$ splits into only finitely many such orbits by the usual method of computing the tangent space. Consequently there are only finitely many nilpotent $G_{x,0:0+}$-orbits in $\Lg_{x,\frac{n}{m}:\frac{n}{m}+}$. The theory of $\mathfrak{sl}_2$-triple works similarly so that for each nilpotent datum, we can complete $e$ to be an $\Lsl$-triple $(e,h,f)\in\Lg_{x,d:d+}\times\Lg_{x,0:0+}\times\Lg_{x,(-d):(-d)+}$.\p

The torus $\uS$ has a reduction $\sS\subset\sG_x$, which is again a maximal ($k$-)split torus. Fix a nilpotent datum $(x,d,e)$ and suppose $x\in\A$. We now define some associated objects. By \cite[Appendix A]{De02b}, there exists some cocharacter $\lambda:\mathbb{G}_m/_k\ra\sG_x$ such that $d\lambda(1)=h$, $\lambda$ acts on $e$ by weight $2$, and $\lambda$ acts on $\bar{\Lg}$ by weight bounded by, say, $\frac{1}{2}\mathrm{char}(k)$. By $G_{x,0:0+}$-conjugation on both $(e,h,f)$ and $\lambda$ we may and shall assume that the image of $\lambda$ is in $\sS$. This determines $\lambda\in X_*(\sS)$ up to the Weyl group $N_{\sG_x}(\sS)/Z_{\sG}(\sS)$. Since $X_*(\sS)\cong X_*(\uS)$, we also think of $\lambda$ as a vector on the apartment $\A$.\p

\begin{definition}\label{assococh} We call $\lambda\in X_*(\sS)\cong X_*(\uS)$ the associated cocharacter of $(x,d,e)$. We write $\Lg_{x,d:d+}=\bigoplus_{i\in\Z}\Lg_{x,d:d+,i}$ where $\Lg_{x,d:d+,i}$ is the subspace on which $\lambda$ acts by weight $i$. We also use the convention $\Lg_{x,d:d+,\le 1}:=\bigoplus_{i\le 1}\Lg_{x,d:d+,i}$, etc.
\end{definition}

Also, \cite[Def. 5.3.4]{De02b} defines a nilpotent orbit $\CO$ in $\Lg$ associated to $(x,d,e)$ (written $\CO(F^*,e)$ there with $x\in F^*\subset\B_G$). It is characterized by the property (see the proof of \cite[Cor. 5.2.4]{De02b}) that for any nilpotent orbit $\CO'$ in $\Lg$, $\CO'\cap e+\Lg_{x,d:d+,\le 1}+\Lg_{x,d+}\not=\emptyset\Leftrightarrow\mathrm{clos}(\CO')\supset\CO$, where $\mathrm{clos}(\CO')$ is the closure in the $p$-adic (analytic) topology.\p

\begin{definition}\label{assoorb} We call $\CO$ the {\bf associated orbit} of $(x,d,e)$.\p
\end{definition}

By \cite[Thm. 5.6.1]{De02b}, every nilpotent orbit in $\Lg$ appears as the associated orbit for some nilpotent datum.\p

\begin{definition}\cite[Def. 3.3.6]{AD02} For $\gamma\in\Lg$,  $\depth(\gamma):=\max\{d\,|\,\gamma\in\Lg_{y,d}\text{ for some }y\}$.
\end{definition}

Here $\depth(\gamma)$ is always defined as a rational number with denominator bounded by some integer determined by $\uG$ unless $\gamma$ is nilpotent, in which case we put $\depth(\gamma)=+\infty$.\p

\begin{definition}
If $\gamma$ is semisimple, contained in the Lie algebra of a maximal torus $T$, one say $\gamma$ is {\bf good} \cite[Def. 5.2]{AR00} if for every absolute root $\alpha\in\Phi(G/_{\bar{F}},T/_{\bar{F}})$, either $d\alpha(\gamma)=0$ or $\text{val}(d\alpha(\gamma))=\text{depth}(\gamma)$. This notion is independent of $T$.\p
\end{definition}

In the next three lemmas let us fix $\gamma\in\Lg$ non-nilpotent and put $d=\depth(\gamma)$.\p

\begin{lemma}\label{JorDec} There exists $\gamma_0\in\Lg$, good of depth $d$, such that $[\gamma,\gamma_0]=0$ and $\gamma':=\gamma-\gamma_0$ has depth greater than $d$.
\end{lemma}

\begin{proof} This follows from applying \cite[Prop. 5.4]{AR00} to the semisimple part of $\gamma$.
\end{proof}

Write $G'=Z_G(\gamma_0)$ the centralizer of $\gamma_0$. This is a twisted Levi subgroup of $G$. Let $\Lg'=\text{Lie }G'$ and $\B_{G'}$ be the (extended) Bruhat-Tits building of $G'$. We fix an embedding $\B_{G'}\hookrightarrow\B_G$, which is canonical up to translation by $X_*(Z(G'))\otimes\mathbb{R}$.\p

\begin{lemma}\cite[Lemma 2.3.3]{KM03}\label{KM} For any $x\in\B_G$ and $\gamma^{\epsilon}\in\Lg_{d+}':=\bigcup_{y\in\B_{G'}}\Lg_{y,d+}'$, we have $\gamma_0+\gamma^{\epsilon}\in\Lg_{x,d}\Rightarrow x\in\B_{G'}$.
\end{lemma}\hp

\begin{lemma}\label{finite} For any $x\in\A$ (resp. $x\in\B_G$), there is a finite subset $W_x\subset N_G(S)$ (resp. $W_x\subset G$) such that for $g\in G$, $\ad(g)\gamma\in\Lg_{x,d}\Rightarrow g\in G_x\cdot W_x\cdot G'$ and $\ad(g)\gamma_0\in\Lg_{x,d}\Leftrightarrow g\in G_x\cdot W_x\cdot G'$.
\end{lemma}\qp

\begin{proof} Fix $\A'\subset\A$ some apartment in $\B_{G'}$. Let $w\in G$ be such that $\ad(w)\gamma\in\Lg_{x,d}$. We have $w^{-1}x\in\B_{G'}$ by Lemma \ref{KM}. By multiplying some elements in $G'$ to the right of $w$ we may assume $w^{-1}x\in\A'\subset\A$. Write $P\subset\A$ the discrete subset of points in $\A$ that are in the same $G$-orbit as $x$. Since $S$ acts on $\A$ by translating via a lattice, $P=\sqcup P_j$ decomposes into finitely many $S$-orbits. Moreover, either $P_j\cap\A'=\emptyset$ or $P_j\cap\A'$ is a single $S'$-orbit. In the latter case pick $w_j\in N_G(S)$ (resp. $w_j\in G$) so that $w_j^{-1}x\in P_j$. Then $W_x=\{w_j\,|\,P_j\cap A'\not=\emptyset\}$ will do.
\end{proof}\p

\p

\section{Algorithm for Shalika germs}\label{secmain}

Recall $F$ is a non-archimedean local field with residue field $k$. Denote by $|\cdot|$ the normalized norm on $F$ so that $|\pi_F|=q^{-1}$ where $\pi_F$ is any uniformizer and $q=\#k$. Fix $\gamma\in\Lg$ non-nilpotent and put $d=\depth(\gamma)$. Based on DeBacker's works \cite{De02a,De02b} we shall compute the Shalika germs at $\gamma$. We adapt the notations in Lemma \ref{JorDec} and the paragraph after it.\p

\begin{definition} We say a function $f\in C_c^{\infty}(\Lg)$ is very smooth of depth $d$ if there exists some $y\in\B_G$ such that $\text{supp}(f)\subset\Lg_{y,r}$ and $f$ is locally constant by $\Lg_{y,r+}$.\p
\end{definition}

\begin{definition} We say a finite set of functions $\{f_i\}_{i\in I}\subset C_c^{\infty}(\Lg)$ {\bf separates nilpotent orbits} if all $f_i$ satisfy (\ref{Shalika}), and the vectors $\{I(\CO,f_i)\}_{\CO\in\CO^G(0)}$ span $\mathbb{C}^{\CO^G(0)}$.\p
\end{definition}

%The following key lemma to this article is a corollary of the works of DeBacker in \cite{De02a} and \cite{De02b}.\p

\begin{lemma}\label{DeBacker} For any $d\in\mathbb{R}$, there exists a finite set of functions, all very smooth of depth $d$, which separates nilpotent orbits.
\end{lemma}\qp

\begin{proof} Any nilpotent orbit $\CO\in\CO^G(0)$ is the associated orbit (Definition \ref{assoorb}) of some nilpotent datum $(x,d,e)$. In particular we have $e+\Lg_{x,d:d+,\le 1}+\Lg_{x,d+}$ meets a nilpotent orbit $\CO'\in\CO^G(0)$ iff $\CO\subset\text{clos}(\CO')$. The function $f_i$ are then taken to be the characteristic function of $e+\Lg_{x,d:d+,\le 1}+\Lg_{x,d+}$, If these $f_i$ satisfy (\ref{Shalika}), we'll have $\{I(\CO,f_i)\}_{\CO,\CO_i\in\CO^G(0)}$ as a strictly upper triangular matrix with non-vanishing diagonal, and thus the column vectors span $\mathbb{C}^{\CO^G(0)}$.\p

It remains to show that the characteristic function of $e+\Lg_{x,d:d+,\le 1}+\Lg_{x,d+}$ satisfies (\ref{Shalika}). One observes that $\Lg_{x,d:d+,\le 1}+\Lg_{x,d+}\supset\Lg_{x+\epsilon\lambda,d}$ for very small $\epsilon>0$ and the associated cocharacter $\lambda$ (Definition \ref{assococh}). By \cite[Thm. 2.1.5]{De02a}, these functions do satisfy (\ref{Shalika}).
\end{proof}\p

We now take the finite set of functions $\{f_i\}_{i\in I}\subset C_c^{\infty}(\Lg)$ constructed in the proof of Lemma \ref{DeBacker} so that each $f_i$ is supported on $\Lg_{x_i,d}$ and locally constant by $\Lg_{x_i,d+}$ for some $x_i\in\B_G$. For each $x_i$, we fix a choice of $W_{x_i}$ from Lemma \ref{finite} where all different $w\in W_{x_i}$ are in different double cosets in $G_{x_i}\backslash G/G'$.\p

To compute the Shalika germs at $\gamma$, it suffices to compute the (normalized) orbital integral $I^G(\gamma,f_i)$; the rest are ``combinatorial numbers'' from $I^G(\CO,f_i)$. We have either the $G$-orbit of $x_i$ does not meet $\B_{G'}$, in which case by Lemma \ref{KM} we have $I^G(\gamma,f_i)=0$, or by conjugating $f_i$ we may assume $x_i\in\B_G$. Now by Lemma \ref{finite} we have

\[I^G(\gamma,f_i)=\sum_{w\in W_{x_i}}\frac{|D^G(\gamma)|^{1/2}\cdot|G_{x,0+}|}{|G'_{w^{-1}x_i}|}\sum_{\dot{g}\in G_{x_i}/G_{x_i,0+}}\int_{G'}f_i(\text{ad}(\dot{g}w)\gamma_0+\text{ad}(\dot{g}wg)(\gamma'))dg\]
\vskip -.2cm
\[=\sum_{w\in W_{x_i}}\frac{|D^G(\gamma)|^{1/2}\cdot|G_{x_i,0+}|}{|D^{G'}(\gamma')|^{1/2}\cdot|G'_{w^{-1}x_i}|}I^{G'}(\gamma',\tilde{f}_{i,w,\gamma_0}').
\]
where $\tilde{f}_{i,w,\gamma_0}':\Lg'\rightarrow\mathbb{C}$ is given by \[\tilde{f}_{i,w,\gamma_0}'(\eta)=\sum_{\dot{g}\in G_{x_i,0:0+}}f_i(\text{ad}(\dot{g}w)(\gamma_0+\eta)).\]

Note $\tilde{f}_{i,w,\gamma_0}'$ is locally constant by $\Lg_{w^{-1}x_i,d+}\cap\Lg'=\Lg'_{w^{-1}x_i,d+}\supset\Lg'_{w^{-1}x_i,\depth(\gamma')}$. In particular by \cite[Thm 2.1.5]{De02a} the Shalika germ expansion at $\gamma'\in\Lg'$ holds for $\tilde{f}_{i,w,\gamma_0}'$, i.e.

\begin{equation}\label{homoeq}I^{G'}(\gamma',\tilde{f}_{i,w,\gamma_0}')=\sum_{\CO'\in\CO^{G'}(0)}\Gamma_{\CO'}(\gamma')I^{G'}(\CO',\tilde{f}_{i,w,\gamma_0}').
\end{equation}\hp

We treat $\Gamma_{\CO'}(\gamma')$ as a black box. As for $I^{G'}(\CO',\tilde{f}_{i,w,\gamma_0}')$, we have
\[I^{G'}(\CO',\tilde{f}_{i,w,\gamma_0}')=\sum_{\dot{g}\in G_{x_i}/G_{x_i,0+}}\int_{\CO'}f_i(\ad(\dot{g}w)(\gamma_0+\eta))d\eta\]
\vskip-.2cm
\[=\sum_{\dot{g}\in G_{x_i}/G_{x_i,0+}}\sum_{\delta\in\Lg_{x_i,d:d+}}f_i(\text{ad}(\dot{g})(\ad(w)\gamma_0+\delta))\int_{\ad(w)\CO'}1_{\delta+\Lg_{x,d+}}(\eta)d\eta.
\]\p

Write $n_w(\delta):=\int_{\ad(w)\CO'}1_{\delta+\Lg_{x,d+}}(\eta)d\eta$ the integral on the RHS of the last equation. They are combinatorial numbers (in terms of $q=\#k$) that can be computed using the method of Ranga Rao \cite{Ra72}. Note that $n_w(\delta)=0$ unless $\delta\in\Lg_{x,d:d+}$ is nilpotent and commutes with the image of $\ad(w)\gamma_0$ in $\Lg_{x_i,d:d+}\subset\bar{\Lg}$ in the sense of Lemma \ref{RY}. There are only finitely many orbits of such $\delta$, and the number of $\delta$ in each such orbit is also a combinatorial number. We then arrive at a sum of the form
\begin{equation}\label{overdelta}
\sum_{\delta}\sum_{\dot{g}\in G_{x_i}/G_{x_i,0+}}\til{n}_w(\delta)f_i(\text{ad}(\dot{g})(\ad(w)\gamma_0+\delta)),
\end{equation}
where $\delta$ runs over representatives of these orbits and $\til{n}_w(\delta)$ is $n_w(\delta)$ times the number of element in the orbit. Recall in the proof of Lemma \ref{DeBacker}, $f_i$ is the characteristic function of some subset of $\Lg_{x_i,d:d+}$. We are thus counting how many $\dot{g}\in G_{x_i}/G_{x_i,0+}$ can conjugate $\ad(w)\gamma_0+\delta$ into $\text{supp}(f_i)$. This counts the number of $k$-points on some variety over $k$ which we'll write out from the definition of $f_i$ in Section \ref{secgeom}.\p

%The quantity left to be understood is
%\[\sum_{\dot{g}\in G_{x_i}/G_{x_i,0+}}\sum_{Z\in\Lg_{x,r:r+}}f_i(\text{ad}(\dot{g})(\ad(w)X_0+Z))n_w(Z).\]

\begin{remark} We collect the ``combinatorial numbers'' that appear. To begin with there is the matrix $\{I^G(\CO,f_i)\}$. Recall that in the construction of Lemma \ref{DeBacker} each $f_i$ is associated to some $\CO_i\in\CO^G(0)$. We should think of $\{I^G(\CO,f_i)\}$ as an (blockwise) upper triangular matrix; in order to derive $\Gamma_{\CO}(\gamma)$ from $I^G(\gamma,f_i)$ we need to invert this matrix. Other combinatorial numbers are: the numbers $\til{n}_w(\delta)$, the ratio $|D^G(\gamma)|^{1/2}/|D^{G'}(\gamma')|^{1/2}$, the measure of $G_{x_i,0+}$ and $G_{w^{-1}x_i}'$, and possibly orders of subgroups of $G_{x_i,0:0+}$. In a sense the sum over different $w\in W_{x_i}$ and the sum over $\delta$ in (\ref{overdelta}) is also combinatorial.\hp

In any case these are ``finite.'' There are only finitely many possible twisted Levi subgroups of $G$ up to conjugation, and finitely many possible depths $d$ modulo $2\Z$, while nilpotent orbital integrals have homogeneity property under $d\mapsto d+2$ and $\Lg_{x,d:d+}\cong\Lg_{x,d+1:(d+1)+}$.\p
\end{remark}

\p
\section{The template}\label{secgeom}

In this section only, $G$ is a connected reductive group over $k$. Let $\theta$ be an order $m$ automorphism of $G\otimes_k\bar{k}$. We assume $\mathrm{char}(k)\gg1$ relative to $G$ and $\mathrm{char}(k)\nmid m$. Denote by $G^{(0)}:=(G^{\theta})^o$ and $\Lg^{(1)}:=\Lg^{(\theta=\zeta_m^n)}$ for a fixed $m$-root of unity $\zeta_m\in\bar{k}^{\times}$. We suppose that $\Lg^{(1)}$ is defined over $k$.\p

Fix $\gamma\in\Lg^{(1)}$. Let $\lambda:\mathbb{G}_m\ra G^{(0)}$ be the associated cocharacter (Definition \ref{assococh}, with the role of $\bar{G}$ replaced by $G$, $\sG$ replaced by $G^{(0)}$ and $\Lg_{x,d:d+}$ replaced by $\Lg^{(1)}$). Let
$G^{(0)}_{\le 0}\subset G^{(0)}$ be the parabolic subgroup whose Lie algebra consists of non-positive weight subspaces for $\lambda$ and $U^{(0)}$ its unipotent radical. Let $\Lg_{\le 1}^{(1)}$ (resp. $\Lg_{\le 2}^{(1)}$) be the sum of all the weight spaces with weight $\le 1$ (resp. $\le2$) in $\Lg^{(1)}$ with respect to $\lambda$. By the same proof of \cite[4.14]{SS70} the $G^{(0)}_{\le 0}$-image of $e+\Lg_{\le 1}$ is dense in $\Lg_{\le 2}^{(1)}/\Lg_{\le 1}^{(1)}$. Let $E\subset G^{(0)}_{\le 0}$ be the subgroup that stabilizes $e+\Lg_{\le 1}^{(1)}\subset\Lg_{\le 2}^{(1)}$ and $E^o$ its identity component. Define
\[\hat{\CH}_e(\gamma):=\{g\in G^{(0)}/E^o\,|\,\ad(g^{-1})(\gamma)\in e+\Lg_{\le 1}^{(1)}\}.\]\qp

This is our template, it has a natural quasi-finite map of generic degree $|\pi_0(E)|$ to the projective variety
$\CH_e(\gamma):=\{g\in G^{(0)}/G^{(0)}_{\le 0}\,|\,\ad(g^{-1})(\gamma)\in\Lg_{\le 2}^{(1)}\}$. As $\lambda$ is unique up to appropriate conjugation, $\hat{\CH}_e(X)$ and $\CH_e(\gamma)$ depends only on $e$ but not on the choice of $\lambda$. \p

\begin{theorem}\label{var} All varieties that appear in the algorithm of the previous section can be described as $\hat{\mathcal{H}}_e(\gamma)$, where $(G,\theta)$ appears as $(\bar{G},\theta)$ in Lemma \ref{RY} and $\gamma$ is not nilpotent.
\end{theorem}\hp

\begin{proof} Comparing with the terminology of the previous section, $m$ is the smallest denominator of $d$. By taking suitable power of $\theta$ we have $\Lg^{(1)}\cong\Lg_{x,d:d+}$ (the latter in the notation of the previous section). The $e+\Lg_{\le 1}^{(1)}$ here was $e+\Lg_{x,d:d+,\le 1}$ in the previous section, which was the support of $f_i$ (see the proof of Lemma \ref{DeBacker}). Our $\gamma$ was $\ad(w)\gamma_0+\delta$ and our $\lambda$ was the $\lambda$ in Definition \ref{assococh} and Lemma \ref{DeBacker}. We take $|E^o(k)|$ as a combinatorial black box. Finally, Lang's theorem implies $G^{(0)}(k)/E^o(k)\cong (G^{(0)}/E^o)(k)$; this is why we need $E^o$ instead of $E$.
\end{proof}\hp

\begin{lemma}\label{smooth} For any $\gamma\in\Lg^{(1)}$, the variety $\hat{\CH}_e(\gamma)$ is smooth.
\end{lemma}

\begin{proof} The quotient $G^{(0)}/E^o$ is obviously smooth, thus it suffices to show that $\hat{\CH}_e(\gamma)$ is a local complete intersection in it. In other words, we want to prove that $\ad(g^{-1})\gamma(\gamma)\in e+\Lg_{\le1}^{(1)}$ impose independent conditions on tangent space. This can be achieved by showing that $[\Lg^{(0)},\gamma]$ intersects with $\Lg_{\le 1}^{(1)}$ transversally. But they do; by properties of $\mathfrak{sl}_2$-triples we have $[\Lg^{(0)},\gamma]\supset[\Lg_{\ge 0}^{(0)},\gamma]=\Lg_{\ge 2}^{(1)}$.
\end{proof}\hp

\begin{remark} This also shows that the image of $\hat{\CH}_e(\gamma)$ in $\CH_e(\gamma)$ is smooth, which is not the case for $\CH_e(\gamma)$. In fact, the varieties $\CH_e(\gamma)$ are examples of Hessenberg varieties of Goresky, Kottwitz and MacPherson \cite{GKM06} except that our varieties are ``singular'' in two ways: our $\gamma$ might no longer be regular semisimple, and consequently $\CH_e(\gamma)$ may not be smooth. Note that $\hat{\mathcal{H}}_{e}(\gamma)$ arises as a family, where $\gamma$ varies in appropriate strata in $\Lg^{(1)}$.
\end{remark}\hp

\begin{remark} Since $\gamma\in\Lg^{(1)}$ is not nilpotent, such representations $G^{(0)}\curvearrowright\Lg^{(1)}$ are classified in \cite{Le09}, \cite{Le13} and \cite{RLYG12}.
\end{remark}\hp

\begin{remark} Hales \cite{Ha94} had essentially found (see the next remark), from the study of stable subregular Shalika germs of classical groups, some $\hat{\CH}_e(\gamma)$ of Hessenberg varieties that admit maps to hyperelliptic curves. In \cite[Sec. 3]{Ts15c} the author identified more examples of $\hat{\CH}_e(\gamma)$ as certain quasi-finite covers of symmetric powers of hyperelliptic curves. \end{remark}\hp

\begin{remark} The stabilizer of $\gamma$ in $G^{(0)}$ acts naturally on $\hat{\CH}_e(\gamma)$. Those representations for which this stabilizer is finite for regular semisimple $\gamma$ often arise in recent advances on arithmetic statistics initiated by Bhargava. In these studies, $\hat{\CH}_e(\gamma)$ or $\CH_e(\gamma)$ usually appear with different descriptions. For example, Thorne \cite{Th13} studied the case $m=2$ for $G$ of type {\bf ADE} and $\theta$ such that $\gamma$ has finite stabilizer in $G^{(0)}$ when $\gamma$ is regular semisimple. One of his result is that when $e$ is subregular nilpotent, he described (in a different setup) the quotient of $\hat{\CH}_e(\gamma)$ by the finite stabilizer as specific plane curve.\p

Following the method of this article, the point-counting of such curves appears in the stable subregular Shalika germs, essentially recovering Hales' result in type {\bf A} and {\bf D}, and imply that in type ${\bf E}$ the stable subregular Shalika germs depends on point-counting of non-hyperelliptic curves of genus $3$ or $4$ over $k$. Fintzen, Li, Shankar and myself had also carried out similar computation for types {\bf B}, {\bf C}, {\bf F}$_4$ and {\bf G}$_2$, along the way recovering Hales' other results on stable subregular Shalika germs.
\end{remark}\hp

There has been the conjecture that orbital integrals for classical groups might only have to do with hyperelliptic curves. The following conjecture, which arose from a discussion with Zhiwei Yun, suggests a precise formulation.\hp

\begin{conjecture}
For $\hat{\CH}_{e}(\gamma)$ as in this section with $G/_{\bar{k}}$ having only isogenous factors of type {\bf A}, {\bf B}, {\bf C}, {\bf D} and {\bf G}$_2$, the (semi-simplification of) $\bar{\Q}_{\ell}[\text{Gal}(\bar{k}/k)]$-modules $H^*_c(\hat{\mathcal{H}}_{e}(\gamma)\otimes\bar{k}\,,\bar{\Q}_{\ell})$ are isomorphic to the direct sum of some submodules of certain tensor products of $H^1_c(\cdot\,,\bar{\Q}_{\ell})$ of hyperelliptic curves and $H^0_c(\cdot\,,\bar{\Q}_{\ell})$ of finite \'{e}tale schemes.
\end{conjecture}\hp

%The following idea benefits from discussion with Alexis Bouthier, Alexei Oblomkov and Zhiwei Yun. With the method in \cite[Sec. 2.3, 2.4]{De02a} and that of this work, the following is expected:\qp

%\begin{conjecture} For any regular semisimmple element in a reductive Lie algebra over $k((t))$ and any parahoric subgroup of the Lie group, the corresponding affine Springer fiber \cite{KL} can be stratified so that each strata is an iterated affine space fibration over $\hat{\mathcal{H}}_{e}(X)$ where $G$, $e$ and $X$ arises in the manner as described in the proof of Theorem \ref{var}. 
%\end{conjecture}\hp

%For example, in the case of $\text{SL}_2$ with split regular semisimple element, one can stratify the affine Springer fiber into $(\mathbb{A}^n-\mathbb{A}^{n-1})$'s for different $n$'s up to the dimension (a point when $n=0$). I am unsure whether for the affine Springer fiber we actually expect to need the variety $\hat{\mathcal{H}}_{e}(X)$, or merely its open image in $\mathcal{H}_{e}(X)$.\p

\p
\section{Formal applications}\label{secapp}

We go back to the notations of Section \ref{secBT} and \ref{secmain}, and continue to assume $\mathrm{char}(k)\gg1$.\p

\begin{corollary} Suppose $\gamma=\gamma_0+\gamma_1+...+\gamma_n+\gamma_{n+1}\in\Lg$ where $[\gamma_i,\gamma_j]=0$, each $\gamma_i$ are good of depth $r_i$ except for $\gamma_{n+1}$ which is nilpotent (could be zero). Here $r_i<r_j$ for $i<j$. Let $G_0:=Z_G(\gamma_0)$, $G_1:=Z_{G_0}(\gamma_1)$, ..., $G_n:=Z_{G_{n-1}}(\gamma_n)$, and $Z_i:=Z(G_i)$ the corresponding centers, so that $Z_0\subset Z_1\subset ...\subset Z_n$ are subtori of $G$. Let $\mathfrak{z}_i$ be the Lie algebra of $Z_i$ and $(\mathfrak{z}_i)_{r_i+}$ its depth-$(r_i+)$ part. Then we have 
\[\Gamma_{\CO}(\gamma+\eta)=\Gamma_{\CO}(\gamma),\;\forall\, \eta\in(\mathfrak{z}_0)_{r_0+}+...+(\mathfrak{z}_n)_{r_n+}.\]
\end{corollary}\qp

\begin{proof}
By the inductive nature of the algorithm in Section \ref{secmain} we only have to prove it for the part about $\gamma_0$ and $\mathfrak{z}_0$. Changing $\gamma_0$ by some $\eta_0\in(\mathfrak{z}_0)_{r_0+}$ does not affect its centralizer in $\Lg$. Moreover, the algorithm in Section \ref{secmain} only makes use of the image of $\gamma_0$ in $(\mathfrak{z}_0)_{r_0:r_0+}$ (when defining function $\til{f}_{i,w,\gamma_0}'$). This proves the corollary.
\end{proof}\hp

\begin{definition} For $\gamma\in\Lg$ whose semisimple part $\gamma_s$ lies in the Lie algebra of a maximal torus $T\subset G$, similar to \cite[Def. 4.1]{AK07} we define singular depth by $s(\gamma):=\max\{\text{val}(d\alpha(\gamma_s))\,|\,\alpha\in\Phi(G/_{\bar{F}},T/_{\bar{F}})\}.$
\end{definition}\hp

$s(\gamma)$ measures how close $\gamma$ is to the singular locus $\Lg\backslash\Lg^{rs}$. For example $\gamma$ is regular and good iff $s(x)=\text{depth}(\gamma)$, and $s(\gamma)=+\infty$ iff $\gamma\not\in\Lg^{rs}$ is not regular semisimple. It was an idea of Hales that Shalika germs should have homogeneity properties near a singular semisimple element in an inductive manner. Indeed we have\hp

\begin{corollary}\label{homo}
For any $\gamma\in\Lg$, suppose $\gamma=\gamma_0+\gamma_1$ with $[\gamma_0,\gamma_1]=0$ and $s(\gamma_0)<\text{depth}(\gamma_1)$. Then for any nilpotent orbit $\CO\in\CO^G(0)$, we have
\[\Gamma_{\CO}(\gamma_0+t^2\gamma_1)=p_{\CO,\gamma_0,\gamma_1}(|t|),\;\forall t\in F^{\times}, |t|\le 1,\]\qp

where $p_{\CO,\gamma_0,\gamma_1}(\cdot)$ is an even polynomial with coefficients in $\Q$ with degree no larger than $\dim G'-\text{rank}_{\bar{F}}G'$, where $G'=Z_G(\gamma_0)$.
\end{corollary}\qp

This corollary follows from (\ref{homoeq}) in the algorithm and the homogeneity property of the (normalized) Shalika germs for $G'$. In fact, the coefficient of $|t|^{2d}$ in $\gamma_{\CO,\gamma_0,\gamma_1}(|t|)$ depends linearly on Shalika germs at $\gamma_1$ of nilpotent orbits in $\Lg'$ whose dimensions are equal to $\dim G'-\text{rank}_{\bar{F}}G'-2d$.\p

As mentioned in the introduction, this suggests the principle that all complications of orbital integral/Shalika germs must appear before $s(\gamma)$ gets too deep. We give an example just to demonstrate this idea.\hp

\begin{corollary} To prove the fundamental lemma for the Lie algebra for arbitrary $\gamma\in\Lg^{rs}$, it suffices to do so for those $\gamma$ for which $s(\gamma)<\text{rank}_{\bar{F}}G\cdot\dim G$.
\end{corollary}

\begin{proof} Any element $\gamma\in\Lg^{rs}$ can be written into a sum $\gamma=\gamma_0+\gamma_1+...+\gamma_n$ of commuting good elements of strictly increasing depths, with $n<\text{rank}_{\bar{F}}G$. Using Corollary \ref{homo}, one can scale the depths of these $\gamma_i$ by even integers, and it suffices to run over no more than $\frac{1}{2}\dim G$ even integers for each to interpolate the polynomials $\gamma_{\CO,\gamma_0,\gamma_1}$ (one also needs to check that the transfer factor behaves under this scaling, etc).
\end{proof}\p

Let us denote by $1_{\Lg_x}\in C_c^{\infty}(\Lg)$ the characteristic function of $\Lg_x$. The author became aware of the following direction of applications thanks to Julia Gordan and Nicolas Templier.\hp

\begin{corollary}\label{bound} Fix $x\in\B_G$. There exists constants $C>0$ and $N\in\Z_{>0}$ such that for any regular semisimple $\gamma\in\Lg$, the normalized orbital integral $I^G(\gamma,1_{\Lg_x})$ is bounded by $q^{C\cdot(\dim G)^N\cdot\mathrm{rank}_{\bar{F}}G}$. The constants $C$ and $N$ are universal (independent of $G$, $F$, $x$, etc). One can take $N=2$.
\end{corollary}\hp

\begin{proof} When we use test functions $f_i$'s to derive the Shalika germs $\Gamma_{\CO}(\gamma)$ from $I^G(\gamma,f_i)$, we need to invert a blockwise triangular matrix $\{I^G(\CO,f_i)\}$ with no more than $\dim G$ blocks ordered by $\dim\CO$. As each $f_i$ is associated to some orbit $\CO_i$, we have $I^G(\CO,f_i)/I^G(\CO_i,f_i)$ bounded by $q^{C_1\dim G}$ for some universal constant $C_1$. Also say the number of nilpontent orbits in $\Lg$ is bounded by $q^{C_2\dim G}$ for some universal $C_2$. So we have to invert a strictly blockwise upper triangular matrix with entries bounded by $q^{C_1\dim G}$ and size bounded by $q^{C_2\dim G}$. Moreover there are no more than $\dim G/2$ blocks, since each block correspond to a dimension of some nilpotent orbit. The entries in the inverse of the blockwise triangular matrix will thus be bounded by $q^{C_0(\dim G)^2}$ for some $C_0$.\p

All other numbers are of smaller order. The number of elements in $W_x$ (Lemma \ref{finite}) or any Weyl group that's ever involved is bounded by $(C_3)^{\dim G}$ for some $C_3$. And say the number of $k$-points on any connected reductive group over $k$ of dimension no more than $\dim G$ (i.e. reductive quotients) is bounded by $q^{C_4\dim G}$. Together this shows that each inductive step introduces numbers of magnitude at most $q^{C_0'(\dim G)^2}$ for some $C_0'$. Finally, the inductive process can be done in $\mathrm{rank}_{\bar{F}}G<\dim G$ steps since the semisimple rank of the twisted Levi subgroup decreases by at least $1$ in each step.
\end{proof}\hp

If we assume there is a $G$-equivariant mock exponential map $\displaystyle\bigcup_{y\in\B_G}\Lg_{y,0+}\xra{\sim}\bigcup_{y\in\B_G}G_{y,0+}$ (satisfying \cite[Hyp. 3.2.1]{De02a} for $r=0+$), then one can obtain the same result for orbital integral on $1_{G_x}$, that is\hp

\begin{corollary} Assuming the existence of a $G$-equivariant (mock) exponential map $\mathsf{e}:\bigcup\Lg_{y,0+}\xra{\sim}\bigcup G_{y,0+}$, then the normalized orbital integral $I^G(g,1_{G_x})$ is also bounded by $q^{C\cdot(\dim G)^N\cdot\mathrm{rank}_{\bar{F}}G}$.
\end{corollary}

\begin{proof}
If $g\in G_{y,0+}$ for some $y\in\B_G$ then pulling back by $\mathsf{e}$ reduces us to the Lie algebra case. Otherwise, we may assume $g\in G_{x,0}$, and write $g=g_0g'$, where $g_0$ has finite prime-to-$p$ order, $g'\in\bigcup G_{y,0+}$, and $g_0g'=g'g_0$. The element $g_0$ should be thought as an analogue of an element good of depth $0$. Let $G'=Z_G(g_0)$. An embedding $B_{G'}\hookrightarrow\B_G$ is established by Landvogt \cite{La00}. One proves analogous statement in Lemma \ref{KM} by following the same proof in \cite[Thm. 2.2.1]{KM03} and begin the first step of the algorithm with $f_i$ replaced by $1_{G_x}$. The rest of the inductive process can be pulled back by $\mathsf{e}$, and is the same as for Shalika germs before.
\end{proof}\hp

In \cite[Sec. 7, and Appendix B]{ST15}, Shin-Templier and Cluckers-Gordon-Halupczok showed (via different methods) that orbital integrals of functions in the spherical Hecke algebra has a uniform bound. While we are not able to say anything here about a non-unit element in the Hecke algebra (something might be along \cite{Ts15b}), we provide a more explicit bound for the unit element.\p

\p
\section{Stratification of affine Springer fibers}\label{secASF}

This section is independent from Section \ref{secmain}, \ref{secgeom} and \ref{secapp}, but parallel in spirit. We continue to use the notations in Section \ref{secBT}, but have $k=\bar{k}$ and $F=k((t))$ (nevertheless see Remark \ref{basefield}). We fix $\gamma\in\Lg$ non-nilpotent. Recall that $\sG_x$ is the algebraic group with $\sG_x(k)\cong G_{x,0:0+}$. Since we assume $k=\bar{k}$ we'll often identify $\sG_x$ with $G_{x,0:0+}$.\p

With Lemma \ref{JorDec} we write $\gamma=\gamma_0+\gamma'$ and put $G'=Z_G(\gamma_0)$. We discuss how we can explicitly modify a bit DeBacker's method in \cite[Sec. 2]{De02a} for the geometric setting and combine with Lemma \ref{finite} to stratify the affine Springer fiber over $\gamma$ and write each stratum as a smooth, \'{e}tale-locally trivial fibration over some slight generalizations of affine Springer fibers for $G'$ over $\gamma'$.\p

To begin with, we recall a setup of DeBacker for descent in \cite[Sec. 2.3, 2.4]{De02a}. Fix a nilpotent datum $(x,d,e)$ in Definition \ref{nildat}. Recall that we had fixed an apartment $\A$ and we assume $x\in\A$. We have an associated cocharacter $\lambda$, with which we have a decomposition $\Lg=\bigoplus_{i\in\Z}\Lg_i$ into weight spaces with respect to $\lambda$. We will also write $\Lg_{\ge j}=\bigoplus_{i\ge j}\Lg_i$, $\Lg_{\not= j}=\bigoplus_{i\not=j}\Lg_i$, etc. The Moy-Prasad lattices at any $y\in\A$ respect this decomposition; $\Lg_{y,d_0}=\bigoplus_{i\in\Z}\Lg_{y,d_0,i}$ for any $d_0$, where $\Lg_{y,d_0,i}:=\Lg_{y,d_0}\cap\Lg_i$.\p

Put $\Lg_{y,d_0:d_0+,i}:=\Lg_{y,d_0,i}/\Lg_{y,d_0+,i}$. Then we have induced canonical isomorphism $\Lg_{y,d_0:d_0+}$ $\cong\bigoplus_{i\in\Z}\Lg_{y,d_0:(d_0)+,i}$, from the definition of Moy-Prasad filtration we have canonical isomorphism $\Lg_{y,d_0:d_0+,i}\cong\Lg_{y-s\lambda,(d_0+i\cdot s):(d_0+i\cdot s)+,i}$. For any $s\in\Q$ write $x'=x-s\lambda$ and $d'=d+2s$. Then from above we have $e\in\Lg_{x,d:d+,2}\cong\Lg_{x',d':d'+,2}\subset\Lg_{x',d':d'+}$; we will think of $e$ as also in $\Lg_{x',d':d'+}$. Moreover, recall the associated cocharacter $\lambda$ was defined through an $\mathrm{sl}_2$-triple $(e,h,f)\in\Lg_{x,d:d+}\times\Lg_{x,0:0+}\times\Lg_{x,(-d):(-d)+}$. We can similarly have $(e,h,f)\in\Lg_{x',d':d'+}\times\Lg_{x',0:0+}\times\Lg_{x',(-d'):(-d')+}$. This gives us\p

\begin{lemma} For any $s\in\Q$, identify $e\in\Lg_{x-s\lambda,(d+2s):(d+2s)+}$ as above. Then $\lambda$ is also the associated cocharacter for $(x-s\lambda,d+2s,e)$.
\end{lemma}\hp

For any nilpotent datum $(x,d,e)$ with associated cocharacter $\lambda$, define
%s_0(x,d,e):=\max\{s\ge0\,|\,\Lg_{x,d+}\subset\Lg_{x-s\lambda,d+2s}\}.
\begin{equation}\label{s0}
s_0(x,d,e):=\min\{s>0\,|\,\Lg_{x-s\lambda,(d+2s):(d+2s)+,<2}\not=0\}.
\end{equation}
\[s_1(x,d,e):=\min\{s>0\,|\,\Lg_{x-s\lambda,(d+2s):(d+2s)+,\not=2}\not=0\text{ or }\Lg_{x-s\lambda,0:0+,\not=0}\not=0\}.\]\p

\begin{lemma}\label{trivial} We have $s_0(x,d,e)=s_1(x,d,e)$.
\end{lemma}

%First we prove that $s_0(x,d,e)=s_1(x,d,e)$. Recall that $\Lg_{x,d+,i}\cong\Lg_{x-s\lambda,(d+i\cdot s)+,i}$. Thus for $s>0$ we have $\Lg_{x,d+,i}\subset\Lg_{x-s\lambda,d+2s,i}$ for $i\ge 2$. On the other hand, for $i<2$ we have $\Lg_{x,d+,i}\subset\Lg_{x-s\lambda,d+2s,i}$ if and only if $\Lg_{x-s\lambda,(d+i\cdot s)+,i}=\Lg_{x-s\lambda,d+2s,i}$, which is exactly true only for those positive $s\le s_1(x,d,e)$.\p

\begin{proof} 
Write $s=s_1(x,d,e)$, $x'=x-s\lambda$ and $d'=d+2s$. It suffices to show that we have $\Lg_{x',d':d'+,i}\not=0$ for some $i<2$. By Lemma \ref{RY}, we can embed $\Lg_{x',d':d'+}\subset\bar{\Lg}$ so that $\bar{\Lg}$ is equipped with some finite order automorphism $\theta$ with which $\Lg_{x',d':d'+}=\bar{\Lg}^{\theta=\zeta}$ is the $\zeta$-eigenspace for some root of unity $\zeta$. Also $\Lg_{x',0:0+}=\bar{\Lg}^{\theta=1}$. Nevertheless, $\lambda\in X_*(\uS)\cong X_*(\sS)$ still give compatible decomposition $\bar{\Lg}=\bigoplus_{i\in\Z}\bar{\Lg}_i$ (this is because $\sS\subset\sG_x\cong\bar{G}^{\theta}$; see the paragraph before Definition \ref{assococh}).\p

What we have to prove is that $\bar{\Lg}^{\theta=\zeta}_i\not=0$ for some $i<2$, given the condition either $\bar{\Lg}^{\theta=\zeta}_i\not=0$ for some $i\not=2$ or $\bar{\Lg}^{\theta=1}_i\not=0$ for some $i\not=0$. The key is that $e\in\bar{\Lg}^{\theta=\zeta}_2$ and $(e,h,f)$ is an $\mathfrak{sl}_2$-triple in $\bar{\Lg}$. Thus for $i\ge 2$, $\ad(e):\bar{\Lg}^{\theta=1}_{i-2}\ra\bar{\Lg}^{\theta=\zeta}_i$ is surjective. We may then assume $\bar{\Lg}^{\theta=1}_i\not=0$ for some $i\not=0$. Since $\bar{\Lg}^{\theta=1}=\Lg_{x',0:0+}$, we can even assume $i<0$. But then using the injectivity of $\ad(e):\bar{\Lg}^{\theta=1}_i\ra\bar{\Lg}^{\theta=\zeta}_{i+2}$ we have $\bar{\Lg}^{\theta=\zeta}_{i+2}\not=0$ with $i+2<2$.
\end{proof}\hp

%The idea is that in the process of ``descent'' in Theorem \ref{Amherst}(ii) below, we will go from $(x,d,e)$ to $(x',d',e)$ with $s=s_0(x,d,e)$, $x'=x-s\lambda$ and $d'=d+2s$.

We now proceed to affine Springer fibers. We define some variants as below.\p

\begin{definition}\label{genASF}
For any nilpotent datum $(x,d,e)$, we put $\X_{x,d,e}$ to be the reduced sub(ind-)scheme of relevant (generalized) affine flag variety such that
\[\X_{x,d,e}(k)=\{g\in G/G_x\,|\,\ad(g^{-1})\gamma\in\Lg_{x,d}\text{ and its image in }\Lg_{x,d:d+}\text{ is in the orbit of }e\}.\]

And we put $\Y_{x,d,e}$ to be the scheme with likewise scheme structure such that
\[\Y_{x,d,e}(k)=\{g\in G/K\,|\,\ad(g^{-1})\gamma\in e+\Lg_{x,d:d+,\le 1}+\Lg_{x,d+}\}\]

where $K:=\{g\in G_x\,|\,\text{image of }g\text{ in }\sG_x\text{ is in }Z_{\sG_x}(f)\sG_{x,<0}\}$.
\end{definition}\hp

Before stating our main theorem, we remark that as $\Lg_{x,0:0+}=\bigoplus_{i\in\Z}\Lg_{x,0:0+,i}$ and $\sG_x\cong G_{x,0:0+}$, we can define parabolic subgroups $\sG_{x,\ge0}\subset\sG_x$ by $\Lie\sG_{x,\ge0}\cong\Lg_{x,0:0+,\ge 0}$. In this way $\sG_{x,>0}$ will be its unipotent radical and $\sG_{x,0}$ its Levi subgroup, and similar for $\sG_{x,\le 0}$ and $\sG_{x,<0}$.\p

\begin{theorem}\label{Amherst} For any nilpotent datum $(x,d,e)$, we have the following results for $\X_{x,d,e}$ and $\Y_{x,d,e}$:\hp

(i)$\;\;$ $\X_{x,d,e}=\emptyset$ when $d\ge\depth(\gamma)$ and $\Y_{x,d,e}=\emptyset$ when $d>\depth(\gamma)$.\hp

(ii)$\;$ When $d<\depth(\gamma)$, put $s:=s_0(x,d,e)$ as in (\ref{s0}), $x'=x-s\lambda$, $d'=d+2s$. Then $\X_{x,d,e}\cong\Y_{x',d',e}$.\hp

(iii) When $d<\depth(\gamma)$, we can stratify $\Y_{x,d,e}$ into finitely many strata, indexed by nilpotent ($\sG_x$-)orbits in $\Lg_{x,d:d+}$ with representatives $e'$, so that each stratum is an \'{e}tale-locally trivial fiber bundle over $\X_{x,d,e'}$ with fiber
\[\CH_{(x,d,e)}(e')=\{g\in\sG_{x}/Z_{\sG_{x}}(f)\cdot\sG_{x,<0}\,|\,\ad(g^{-1})e'\in e+\Lg_{x,d:d+,\le 1}\}.\]\hp

(iv) When $d=\depth(\gamma)$, we can stratify $\Y_{x,d,e}$ into finitely many strata, indexed by $\{h\in G'\bsl G/G_x\,|\,h\cdot x\in\B_{G'}\}$ and nilpotent orbits with representative $e'\in\Lg_{h\cdot x,d:d+}'$, so that each stratum is an \'{e}tale-locally trivial fiber bundle over $\X_{h\cdot x,d,e'}^{G'}$. Let $\bar{\gamma}_0$ be the image of $\gamma_0$ in $\Lg_{h\cdot x,d:d+}$. We can put $\bar{\gamma}:=\ad(h^{-1})(\bar{\gamma}_0+e')\in\Lg_{x,d:d+}$.
Then the fiber is
\[\CH_{(x,d,e)}(\bar{\gamma})=\{g\in\sG_x/Z_{\sG_{x}}(f)\cdot\sG_{x,<0}\,|\,\ad(g^{-1})\bar{\gamma}\in e+\Lg_{x,d:d+,\le 1}\}.\]
\end{theorem}\hp

\begin{proof} (i) This follows because when $d>\depth(\gamma)=\depth(\ad(g^{-1})\gamma)$, we cannot have $\ad(g^{-1})\gamma\in\Lg_{x,d}$. When $d=\depth(\gamma)$, even if $\ad(g^{-1})\gamma\in\Lg_{x,d}$, its image in $\Lg_{x,d:d+}$ cannot be nilpotent, for if $\ad(g^{-1})\gamma$ has image $e'\in\Lg_{x,d:d+}$ nilpotent, by putting $\lambda'$ to be the associated cocharacter of $(x,d,e')$, we have $\ad(g^{-1})\gamma\in\Lg_{x+\epsilon\lambda',d-2\epsilon}$ for small positive $\epsilon$.\p

(ii) Let us put $K_1:=\{g\in G_x\,|\,\text{image of }g\text{ in }\sG_x\text{ is in }Z_{\sG_x}(e)\}$ and also $K_2:=\{g\in G_x\,|\,\text{image of }g\text{ in }\sG_x\text{ is in }Z_{\sG_x}(e)\cdot(\sG_x)_{>0}\}$.
Then
\[\X_{x,d,e}(k)=\{g\in G/K_1\,|\,\ad(g^{-1})\gamma\in e+\Lg_{x,d+}\}\]
\[\cong\{g\in G/K_2\,|\,\ad(g^{-1})\gamma\in e+\Lg_{x,d:d+,\ge 3}+\Lg_{x,d+}\}.\]\hp

The latter isomorphism follows because for every $\eta\in e+\Lg_{x,d:d+,\ge 3}\subset\Lg_{x,d:d+}$, there exists a unique $\bar{g}\in (\sG_x)_{>0}\cdot Z_{\sG_x}(e)/Z_{\sG_x}(e)$ such that $\ad(\bar{g}^{-1})\eta=e\in\Lg_{x,d:d+}$; this is a property that holds for any $\Lsl$-triple $(e,h,f)$.\p

Using in Lemma \ref{trivial} that $s=s_0(x,d,e)=s_1(x,d,e)$ we have \begin{equation}\label{id1}e+\Lg_{x,d:d+,\ge 3}+\Lg_{x,d+}=e+\Lg_{x',d':d'+,\le 1}+\Lg_{x',d'+}\subset\Lg_{x,d}\cap\Lg_{x',d'}
\end{equation}
and also
\[K_2=\{g\in G_x\,|\,\text{image of }g\text{ in }G_{x,0:0+}\text{ is in }Z_{\sG_x}(e)\cdot\sG_{x,>0}\}\]
\[\cong\{g\in G_x\,|\,\text{image of }g\text{ in }G_{x,0:0+}\text{ is in }Z_{\sG_x}(e,h,f)\cdot\sG_{x,>0}\}\]
\[\cong\{g\in G_{x'}\,|\,\text{image of }g\text{ in }G_{x',0:0+}\text{ is in }Z_{\sG_{x'}}(e,h,f)\cdot\sG_{x',<0}\}\]
\vskip -.55cm
\begin{equation}\label{id2}
\cong\{g\in G_{x'}\,|\,\text{image of }g\text{ in }G_{x',0:0+}\text{ is in }Z_{\sG_{x'}}(f)\cdot\sG_{x',<0}\}.
\end{equation}\hp

We explain the above three isomorphisms. For the first isomorphism, note that by the property of $\mathfrak{sl}_2$-triples, $e$ can only be centralized by elements of non-negative weights; $Z_{\sG_x}(e)=Z_{\sG_{x,0}}(e)\cdot Z_{\sG_{x,>0}}(e)$. But by definition $Z_{\sG_{x,0}}(e)=Z_{\sG_x}(e,h)$, and $Z_{\sG_x}(e,h)=Z_{\sG_x}(e,h,f)$ since $e$ and $h$ determine $f$ for an $\mathfrak{sl}_2$-triple. The third isomorphism follows similarly. For the second isomorphism, note that since $s=s_0(x,d,e)=s_1(x,d,e)$, we have $\Lg_{x-s'\lambda,0:0+\not=0}=0$ for all $0<s'<s$. Consequently the preimage of $\sG_{x,>0}$ in $G_x$ is the same as the preimage of $\sG_{x',<0}$ in $G_{x'}$ for $x'=x-s\lambda$.\p

In conclusion, by using (\ref{id1}) and (\ref{id2}) we have
\[\X_{x,d,e}(k)=\{g\in G/K_2\,|\,\ad(g^{-1})\gamma\in e+\Lg_{x',d':d',\le 1}+\Lg_{x',d'+}\}=\Y_{x',d',e}(k).\]\p

(iii) Put $\X_{x,d}$ the usual affine Springer fiber with $\X_{x,d}(k)=\{g\in G/G_x\,|\,\ad(g^{-1})\gamma\in\Lg_{x,d}\}$. There is natural map $\Y_{x,d,e}\ra\X_{x,d}$. We claim that the latter can be stratified into locally closed strata $\X_{x,d,e'}$, where $e'$ run over nilpotent orbits in $\Lg_{x,d:d+}$. The claim is equivalent to saying that if $\ad(g^{-1})\gamma\in\Lg_{x,d}$, then its image in $\Lg_{x,d:d+}$ has to be nilpotent. This is more or less in \cite[Sec. 3.2]{AD02}. Nevertheless I learn the following proof from Yiannis Sakellaridis:\hp

\begin{proof}[Proof of claim]
By definition $\gamma\in\Lg_{x',\depth(\gamma)}$ for some $x'\in\B_G$. After a tame base change we may assume $x$ and $x'$ are hyperspecial, and $d$ and $\depth(\gamma)$ are integers. By scaling we may assume $d=0$, so $\depth(\gamma)\in\Z_{>0}$. We may even furthermore assume that $x$ and $x'$ are in the same $G$-orbit, so without loss of generality we may assume $x'=x$.\hp

Let $\bar{A}:=k[\Lg_{x,0:0+}]^{\sG_x}$ be the ring of invariant functions on $\Lg_{x,0:0+}$. Then we can identify $A:=\bar{A}\otimes_kF$ as the ring of invariant functions on $\Lie\uG$ and $A_0:=\bar{A}\otimes_k\CO_F$ the subring of invariant functions on $\Lg_{x,0}$ as a scheme over $\CO_F$. Since $\gamma\in\Lg_{x,\depth(\gamma)}\subset\Lg_{x,1}$, all $f\in A_0$ have to take values in $\pi_F\CO_F$ at $\gamma$. But this implies that all $\bar{f}\in\bar{A}$ takes zero value on the image of $\ad(g^{-1})\gamma$ in $\bar{\Lg}=\Lg_{x,d:d+}$, i.e. the image of $\ad(g^{-1})\gamma$ in $\Lg_{x,d:d+}$ is nilpotent.
\end{proof}\hp

The stratification on $\Y_{x,d,e}$ is then obtained by pulling back, and one checks that the fibers of $\Y_{x,d,e}\ra\X_{x,d}$ on each stratum $\X_{x,d,e'}$ are as asserted. It remains to show that this is an \'{e}tale-locally trivial fibration, which we left to Proposition \ref{fibra} latter in this section.\p

(iv) When $d=\depth(\gamma)$, by Lemma \ref{KM}, $\ad(g^{-1})\gamma\in\Lg_{x,d}\Leftrightarrow\gamma\in\Lg_{g\cdot x,d}\Rightarrow g\cdot x\in\B_{G'}$. Hence such $g$ lies in some double cosets among $\{h\in G'\bsl G/G_x\,|\,h\cdot x\in\B_{G'}\}$. This is finite by Lemma \ref{finite}. It was also shown that representatives of $h$ can be chosen in $N_G(S)$. Fix a choice of such representatives.\p

The scheme $\Y_{x,d,e}$ can thus be decomposed into finitely many $[\Y_{x,d,e}]_h$, so that different parts are disconnected, by putting 
\[[\Y_{x,d,e}]_h(k)=\{g\in G'\cdot h\cdot G_x/K\,|\,\ad(g^{-1})\gamma\in e+\Lg_{x,d:d+,\le 1}+\Lg_{x,d+}\}\]

where $K\subset G_x$ consist of those elements whose image in $\sG_x$ are in $Z_{\sG_x}(f)\cdot(\sG_x)_{<0}$. Note that $G'\cdot h\cdot G_x=G'\cdot G_{h\cdot x}\cdot h$, and we have composition of maps \[\pi:G'\cdot h\cdot G_x/K\twoheadrightarrow G'\cdot h\cdot G_x/G_x\cong G'\cdot G_{h\cdot x}/G_{h\cdot x}\cong G'/G_{h\cdot x}'.\]\hp

From the definition, we similarly verify $\ad(g^{-1})\gamma\in\Lg_{x,d}\Leftrightarrow\ad(\pi(g)^{-1})\gamma\in\Lg_{h\cdot x,d}'\Leftrightarrow\ad(\pi(g)^{-1})\gamma'\in\Lg_{h\cdot x,d}'$. The last equivalence is because $\gamma_0$ is in the center of $\Lg'$. This gives a natural map from $[\Y_{x,d,e}]_h$ to $\X_{h\cdot x,d}^{G'}$, where the latter is an affine Springer fiber for $G'$ over $\gamma'$. Nevertheless $\X_{h\cdot x,d}^{G'}$ can be stratified into $\X_{h\cdot x,d,e'}^{G'}$ indexed by nilpotent orbits $e'\in\Lg_{h\cdot x,d:d+}'$, and one similarly checks that the natural map from $[\Y_{x,d,e}]_h\ra\X_{h\cdot x,d}^{G'}$ restricted to stratum $\X_{h\cdot x,d,e'}^{G'}$ has the asserted fiber. We are thus done with Proposition \ref{fibra} below.\end{proof}\hp

%\begin{lemma}\label{Zariski} Let $G$ be an affine algebraic group over $k$ and $H$ be a connected subgroup. Then the natural projection $G\ra G/H$ is Zariski-locally trivial.
%\end{lemma}

\begin{lemma}\label{etale} Let $G$ be an affine algebraic group over $k$ and $H$ be a subgroup. Then the natural projection $G\ra G/H$ is \'{e}tale-locally trivial.
\end{lemma}\hp

\begin{proposition}\label{fibra}
The maps $\Y_{x,d,e}\ra\X_{x,d}$ in Theorem \ref{Amherst}(iii) and $[\Y_{x,d,e}]_h\ra\X_{h\cdot x,d}^{G'}$ in \ref{Amherst}(iv), when restricted to $\X_{x,d,e'}$ and $\X_{h\cdot x,d,e'}^{G'}$ respectively, is a \'{e}tale-locally trivial fibration.
\end{proposition}

\begin{proof} We will prove the case for $\Y_{x,d,e}\ra\X_{x,d}$ restricted to $\X_{x,d,e'}$; the proof in the other case is identical. Put $\CZ_x$ the generalized affine flag variety with $\CZ_x(k)=G/G_{x,0+}$ and define $\CZ_{x,d,e}$ to be the reduced sub(ind-)scheme with (see Definition \ref{genASF})
\[\CZ_{x,d,e}=\{g\in G/G_{x,0+}\,|\,\ad(g^{-1})\gamma\in e+\Lg_{x,d:d+,\le 1}+\Lg_{x,d+}\}.\]\hp

It suffices to show that $\CZ_{x,d,e}\ra\X_{x,d}$ restricted to $\X_{x,d,e'}$ has \'{e}tale-local sections. Note that this map comes from the natural map $\CZ_x\ra\X_x$, where $\X_x$ is the usual affine flag variety; $\X_x(k)=G/G_x$. By chasing through the definition of $\CZ_x$ and $\X_x$ as ind-scheme constructed via affine Schubert cells as subschemes of Grassmannians, one can show Lemma \ref{etale} implies that the map $\CZ_x\ra\X_x$ is \'{e}tale-locally trivial.\p

Now for the moment denote by $\CZ^*$ the preimage of $\X_{x,d,e'}$ under $\CZ_x\ra\X_x$. Then there is a natural map $\pi:\CZ^*\ra\Lg_{x,d:d+}$ by $\pi(g)=\ad(g^{-1})\gamma$, and $\CZ_{x,d,e}$ is the preimage of $e+\Lg_{x,d:d+,\le 1}+\Lg_{x,d+}$. With abuse of language let us write $s:\X_{x,d,e'}\ra\CZ^*$ an \'{e}ale-local section. Then $\pi\circ s$ has image in the orbit of $e'$, which is isomorphic to $\sG_x/Z_{\sG_x}(e)$. By Lemma \ref{etale} choose $s':\sG_x/Z_{\sG_x}(e)\ra\sG_x$ an \'{e}tale-local section. Then $s'\circ\pi\circ s$ is a map on an \'{e}tale neighborhood of $\X_{x,d,e'}$ to $\sG_x$.\p

Consider the \'{e}tale-local section $s'':\X_{x,d,e'}\ra\CZ^*$ by $s''(g)=s(g)\cdot\left((s'\circ\pi\circ s)(g)\right)^{-1}$. From the definition of $s''$ one has $\pi(s''(g))=id$. But this implies $s''$ trivializes the condition that defines $\CZ_{x,d,e}$ in $\CZ^*$; i.e. if we pull back to the \'{e}tale neighborhood of $\X_{x,d,e'}$ on which $s''$ is defined, then $\CZ_{x,d,e}\ra\X_{x,d}$ becomes a trivial fibration map with the fiber stated in Theorem \ref{Amherst}(iii).
\end{proof}\p

\begin{remark} Theorem \ref{Amherst} provides an algorithm to partially understand $\X_{x,d,e}$ from $\X_{x',d',e'}^{G'}$ for some twisted Levi subgroup $G'$. By applying this algorithm a finite number of times one arrives at some affine Springer fiber for tori which is simply the cocharacter lattice.
\end{remark}\hp

%\begin{remark} The fibers described in (iii) and (iv) of Theorem \ref{Amherst} are the images of $\hat{\CH}_e(\gamma)$ into $\CH_e(\gamma)$ in Section \ref{secgeom}, with $\gamma=e'$ and $\bar{\gamma}$ respectively. By Lemma \ref{smooth} they are smooth; in particular the fibrations in Theorem \ref{Amherst}(iii), (iv) are smooth. Consequently, this algorithm also provides an explicit way to stratify an affine Springer fiber into smooth locally closed subschemes.
%\end{remark}\hp

We say a bit more about the fibers in the theorem. They have the following properties:

\begin{lemma} (i) If the fiber in Theorem \ref{Amherst}(iii)
\[\CH_{(x,d,e)}(e')=\{g\in\sG_{x}/Z_{\sG_{x}}(f)\cdot\sG_{x,<0}\,|\,\ad(g^{-1})e'\in e+\Lg_{x,d:d+,\le 1}\}.\]
is non-empty, then the associated orbit (Definition \ref{assoorb}) of $(x,d,e)$ is contained in the analytic closure of the associated orbit of $(x,d,e')$.\hp

(ii) $\CH_{(x,d,e)}(e')$ is smooth. The same is true for $\CH_{(x,d,e)}(\bar{\gamma})$ in (iv).\hp

(iii) If $e'=e$, then $\CH_{(x,d,e)}(e)$ is isomorphic to an affine space of dimension equal to $\dim Z_{\sG_{x,>0}}(e)$.
\end{lemma}\hp

\begin{proof}(i) By the defining property of associated orbits, there exists some nilpotent element $N\in\Lg_{x,d}$ lifting $\ad(g^{-1})e'$ so that $N$ is contained in the associated orbit of $(x,d,e')$. But then that the image of $N$ lies in $e+\Lg_{x,d:d+,\le 1}$ says that the closure of the orbit of $N$ contains the associated orbit of $(x,d,e)$.\p

(ii) This is essentially Lemma \ref{smooth}; we're looking at the image of $\hat{\CH}_e(\gamma)$ into $\CH_e(\gamma)$ in Section \ref{secgeom}. Nevertheless we reproduce the proof with the notation here. First note that we have
\[\CH_{(x,d,e)}(e')\cong\{g\in\sG_{x}/Z_{\sG_{x}}(e,h,f)\,|\,\ad(g^{-1})e'\in e+Z_{\Lg_{x,d:d+}}(f)\}.\]
This is a property of $\mathfrak{sl}_2$-triples; one can conjugate an element in $e+\Lg_{x,d:d+,\le1}$ into $e+Z_{\Lg_{x,d:d+}}(f)$ using a unique element in $\sG_{x,<0}$. For smoothness it suffices to show that the orbit of $e'$ intersect $e+Z_{\Lg_{x,d:d+}}(f)$ transversally. This can be checked by computing tangent spaces; without loss of generality we may assume $e'\in e+Z_{\Lg_{x,d:d+}}(f)$ and check at $e'$, for which we only need $\Lg_{x,d:d+}=\ad(\Lg_{x,0:0+})(e')+Z_{\Lg_{x,d:d+}}(f)$. This last equality follows from properties of representations of $\mathfrak{sl}_2$. The same proof goes when $e'$ is replaced by $\bar{\gamma}$ in (iv).\p

(iii) Now $e'=e$. We claim that if $e'':=\ad(g^{-1})e\in e+Z_{\Lg_{x,d:d+}}(f)$, then $e''=e$. Recall that $e,h,f$ and $\Lg_{x,d:d+}$ all can be embedded into a Lie algebra $\bar{\Lg}$ by Lemma \ref{RY}, and we can work in $\bar{\Lg}$. Suppose otherwise $e''\not=e$, then with the associated cocharacter $\lambda:\mathbb{G}_m\ra\bar{G}$ ($\Lie\bar{G}=\bar{\Lg}$), we have $t^{-2}\ad(\lambda(t))(e'')$ is a curve in the intersection of $e+Z_{\bar{\Lg}}(f)$ and the orbit of $e$. However, in the proof of (ii) we see that $e+Z_{\bar{\Lg}}(f)$ and the orbit of $e$ intersect transversally, and by dimension computing one sees that they intersect in dimension zero. This yields a contradiction.\p

With the claim, we see that 
$\displaystyle\CH_{(x,d,e)}(e)\cong\{g\in\sG_{x}/Z_{\sG_{x}}(e,h,f)\,|\,\ad(g^{-1})e=e\}\cong Z_{\sG_{x,>0}}(e)$.\end{proof}\p

\begin{remark}\label{basefield}
If instead $k$ is a finite field (or other perfect field that is not algebraically closed), then we might not be able to choose representatives defined over $k$ for our nilpotent orbits. Nevertheless, one checks that the stratification is canonical, and thus the algorithm still works over $k$ while we use the Bruhat-Tits theory for $\bar{k}$ and $F=\bar{k}((t))$. The Frobenius (or $\mathrm{Gal}(\bar{k}/k)$ in general) acts non-trivially on Bruhat-Tits data, and will act in a more subtle manner on the components of our strata.
\end{remark}\hp

\begin{remark} As remarked before Definition \ref{assococh}, our $\lambda$ in Theorem \ref{Amherst}(ii) is chosen up to the Weyl group $N_{\sG_x}(\sS)/Z_{\sG}(\sS)$ at $x$, at least so if we insist on having $\lambda$ and $x$ on the apartment $\A$ corresponding to the pre-chosen maximal split torus $\uS$.
\end{remark}\hp

\p
\bibliographystyle{amsalpha}
\bibliography{biblio}

\end{document}